\documentclass[a4paper,12pt]{amsart}

\usepackage{float}
\usepackage{euscript,eufrak,verbatim}
\usepackage{graphicx}
\usepackage{amscd,pifont}
\usepackage[usenames]{color}
\usepackage[colorlinks,linkcolor=red,anchorcolor=blue,citecolor=blue]{hyperref}
\usepackage{bbm}
\usepackage{amsmath}
\usepackage{amsthm}
\usepackage[all]{xy}
\usepackage{amssymb}

\newtheorem{thm}{Theorem}[section]
\newtheorem{prop}[thm]{Proposition}

\newtheorem{lem}[thm]{Lemma}

\newtheorem{conj}{Conjecture}[section]
\newtheorem{rem}{Remark}[section]

\makeatletter 
\@addtoreset{equation}{section}
\makeatother  

\def\ZN{\mathbb{Z}_{N}}

\def\Z{\mathbb{Z}}
\def\N{\mathbb{N}}
\def\ZZ{\mathcal{Z}}

\def\Zpp{\mathbb{Z}_{p^n}\times \mathbb{Z}_{p^m}}
\def\Zptp{\mathbb{Z}_{p^2}\times \Z_{p}}
\def\vd{{\rm\mathbf{d}}}

\title[spectral set conjecture on $\mathbb{Z}_{p^2}\times \Z_{p}$]{Equi-distributed property and spectral set conjecture on $\mathbb{Z}_{p^2}\times \Z_{p}$}
\author{Ruxi Shi \\
}
\address{Department of Mathematical Sciences, P.O. Box 3000, 90014 
	University of Oulu, Finland}
\email{ruxi.shi@oulu.fi}
\begin{document}

\maketitle

\begin{abstract}
In this paper, we show an equi-disctributed property in $2$-dimensional finite abelian groups $\Zpp$ where $p$ is a prime number. By using this equi-disctributed property, we prove that Fuglede's spectral set conjecture holds on groups $\mathbb{Z}_{p^2}\times \Z_{p}$, namely, a set in $\mathbb{Z}_{p^2}\times \Z_{p}$ is a spectral set if and only if it is a tile. 
\end{abstract}



\section{Introduction}
Let $G$ be a locally compact abelian group. Let $\widehat{G}$ be the dual group of $G$ which consists of all continuous group characters. Let $\Omega$ be a Borel measurable subset in $G$ with positive finite Haar measure.
The set $\Omega$ is called {\em a spectral set} if there exists a set $\Lambda \subset \widehat{G}$ which is
an Hilbert basis of $L^2(\Omega)$, and that $\Omega$ is  called 
{\em a tile}  of $G$ by translation  if there exists a set $T \subset G$ of translates
such that almost surely $\Omega\oplus T=G$. In 1974, Fuglede \cite{f}  formulated the following so-called spectral set conjecture concerning the equivalence relation between tiles and spectral sets: 

\vspace{4pt}
{\em 	A Borel subset $\Omega\subset G$ of positive and finite Haar measure is a spectral set if and only if it is a tile.} 
\vspace{4pt}

\noindent Fuglede's conjecture has
attracted substantial attention over the last decades, especially in the setting $G=\mathbb{R}^d$. Many
positive results towards the conjecture were established before Tao \cite{t} who disproved the conjecture by showing that the
direction ``Spectral $\Rightarrow$ Tiling" is not tenable for $\mathbb{R}^d$ when $d\ge 5$.
Now it is known that the conjecture does not hold in both directions for $\mathbb{R}^d, d \geq 3$ \cite{fmm,k,km,m}.   But the conjecture is still open in lower dimensions $\mathbb{R}$ and $\mathbb{R}^2$.

In its general setting, few is known about this conjecture, even for finite groups. Hence, the question becomes for which group $G$, the spectral set conjecture holds. For $N\ge 1$, the ring of integers modulo $N$ is denoted by $\ZN=\mathbb{Z}/N\mathbb{Z}$. We only know that Fuglede's spectral set conjecture holds on $\mathbb{Z}_{p^n}$ \cite{l2, ffs}, $\mathbb{Z}_p \times \mathbb{Z}_p$ \cite{imp}, $\mathbb{Z}_{p^nq}$ with $n\ge 1$ \cite{mk}, $p$-adic field $\mathbb{Q}_p$ \cite{ffs, ffls}, $\Z_{pqr}$ with $p,q,r$ different primes  \cite{shi} and very recently $\mathbb{Z}_{p^nq^2}$ with $n\ge 1$ \cite{kmsv}. Nevertheless, we either know few about for which group, the spectral set conjecture fails. We only know that the direction ``spectral sets $\Rightarrow$ tiles" fails on  $\mathbb{Z}_3^6$ \cite{t}, $\Z_8^3$ \cite{km}, $\Z_p^4 \quad ($when $p$ prime and $p \equiv 1 \mod 4)$ and $\Z_p^5 \quad ($when $p$ prime and $p \equiv 3 \mod 4)$ \cite{a}; the direction ``tiles $\Rightarrow$ spectral sets " fails on $(\Z/6\Z)^5$ \cite{mk},  $(\Z/6\Z)^3\times \Z_{6p}$, where $p$ is prime with gcd$(6,p)=1$ \cite{fg} and $(\Z/24\Z)^3$ \cite{fmm}. Actually, these non-tile spectral sets (resp. non-spectral tiles) in finite abelian groups lead to non-tile spectral sets (resp. non-spectral tiles) in euclidean spaces. We remark here that we still don't know whether there exists non-spectral tiles in $\Z_{p^n}^d$ for $d\ge 1$ and $p$ prime. The following conjecture is therefore proposed:

\begin{conj}
Let $p$ be a prime. For all $n,d\ge 1$, the tiles in the groups $\Z_{p^n}^d$ are always spectral sets.
\end{conj}

As we have mentioned eariler, the known non-tile spectral sets on $\mathbb{R}^d$ ($d\ge 3$) are constructed by using the non-tile spectral sets on $\Z_{p^n}^d$ for some $n$ (see \cite{fmm,k,km,m}). Since Fuglede's spectral set conjecture is widely open on $\mathbb{R}^2$, it is interesting to know whether Fuglede's spectral set conjecture holds on $2$-dimensional finite abelian groups $\Z_{p^n}^2$ for all $n$. In fact, if Fuglede's conjecture fails on $\mathbb{R}^2$, then the counterexample is probably constructed by a counterexample on $\Z_{p^n}^2$ for some $n$.  We don't know anything about the conjecture on these groups except that Iosevich, Mayeli and Pakianathan \cite{imp} proved that the spectral set conjecture is true on $\Z_p\times \Z_p$.

In this paper, our main interest is Fuglede's spectral set conjecture on $\Zpp$ for $n, m\ge 1$. We first show an equi-distribued property in $\Zpp$. Then we focus on the groups $\mathbb{Z}_{p^2}\times \Z_{p}$. Our main result is the following.
\begin{thm}\label{main thm}
	A set in  $\mathbb{Z}_{p^2}\times \Z_{p}$ is a spectral set if and only if it is a tile of $\mathbb{Z}_{p^2}\times \Z_{p}$ by translation.
\end{thm}

Haessig et al. \cite{hiprv} proved that for $A\subset \Z_p^d$ and $l\in \Z_p^d$, the equation $\widehat{1}_A(l)=0$ if and only if $A$ is equi-distributed on the $p$ hyperplanes $H_t=\{x: x\cdot l=t \}$ for $t\in \Z_p^d$. Actually the $p$ hyperplanes $H_t$ (for $t\in \Z_p^d$) form a partition of $\Z_p^d$. Iosevich, Mayeli and Pakianathan \cite{imp} used this property to prove Fuglede's spectral set conjecture valid on $\Z_p \times \Z_p$. But such property is not true in general or over other fields. Thus we need other kind of equi-distributed property in general case. In this paper, we prove such equi-distributed property in $\mathbb{Z}_{p^n}\times \Z_{p^m}$ in Lemma \ref{lem:hyperplane}. It is the main tool to prove Theorem \ref{main thm}. Actually, Lemma \ref{lem:hyperplane} shows us that it is not equi-distributed on each set of a certain partition of $\mathbb{Z}_{p^n}\times \Z_{p^m}$, but that it is equi-distributed in some chosen planes. This is the main difference between the equi-distributed property in $\mathbb{Z}_{p^n}\times \Z_{p^m}$ and the one in $\Z_p\times \Z_p$.

The article is organized as follows. In Section \ref{sec:Preliminary}, we discuss some basic properties of spectral sets and tiles in finite abelian groups. In Section \ref{sec:Equi-distribution on planes}, we prove an equi-distributed property for a set in $\Zpp$ whenever $\ZZ_A$ is not empty. In Section \ref{sec:Compatible values}, we define the compatible value and then discuss some properties of it. In Section \ref{sec:Tiles implying Spectral sets}, we prove the ``Tiles $\Rightarrow$ Spectral sets" part of Theorem \ref{main thm}. In Section \ref{Spectral sets implying Tiles}, we prove the ``Spectral sets $\Rightarrow$ Tiles" part of Theorem \ref{main thm}.


\subsection*{Notation}

\begin{itemize}
	\item  For $N\in \N$, we denote by $\Z_N$ the cyclic group $\Z/N\Z$ or sometimes the finite set $\{0,1, \dots, N-1 \}$.
	\item For a finite abelian group $G$, denote by $G^*$ the set of all invertible elements in $G$.
\end{itemize}

\section{Preliminary}\label{sec:Preliminary}
In this section, we discuss some basic properties spectral sets and tiles in finite groups.

Let $G$ be a finite abelian group. Let $\widehat{G}$ be the dual group of $G$. It is well-known that $\widehat{G}$ is also a finite abelian group and furthermore isomorphic to $G$. Consequently, we can write 
$$
\widehat{G}=\{\chi_g: g\in G \}
$$
with the commutative multiplicity 
$$
\chi_g \cdot \chi_{g'}=\chi_{g+g'},~\forall~g,g'\in G.
$$
\subsection{Spectral sets in $G$}

We recall that a subset $A\subset G$ is said to be \textit{spectral} if there is a set $B\subset G$ such that 
$$
\{\chi_b: b\in B \}
$$
forms an orthogonal basis in $L^2(A)$. In such case, the set $B$ is called a \textit{spectrum} of $A$ and the pair $(A,B)$  called a \textit{spectral pair}. Since the dimension of $L^2(A)$ is $\sharp A$, the pair $(A ,B)$ being a spectral pair is equivalent to that $\sharp A=\sharp B$ and meanwhile the set $\{\chi_b: b\in B \}$ is orthogonal in $L^2(A)$, that is to say, 
\begin{equation}\label{spectral condition}
\sharp A=\sharp B ~\text{and}~\sum_{a\in A} \chi_{b}\cdot \overline{\chi_{b'}}(a)=0, ~\forall b,b'\in B, b\not=b'.
\end{equation}
This means that the complex matrix $M=\left(\chi_b(a)\right)_{b\in B, a\in A}$ is a complex Hadamard matrix, i.e. $M\overline{M}^{T}=(\sharp A)I$ where $M^T$ is the transpose of $M$ and $I$ is the identity matrix. Since $\sharp A=\sharp B$, it follows that $\overline{M}^TM=(\sharp B)I$. This deduces that $B$ is a spectral set with spectrum $A$.

Denote by $\mathcal{Z}_A=\{g\in G: \widehat{1}_A(g)=0 \}$ the set of zeros of the Fourier transform of the indicator function $1_A$. Since 
$$\widehat{1}_A(g)=\sum_{a\in A} \overline{\chi_g}(a), $$
we restate the above equivalent definition of spectral sets as follows. 
\begin{prop}\label{Equivalence of spectral set}
	Let $A, B\subset G$. The following are equivalent.
	\begin{itemize}
		\item [(1)] $(A,B)$ is a spectral pair.
		\item [(2)] $(B,A)$ is a spectral pair.
		\item [(3)] $\sharp A=\sharp B$; $(B-B)\setminus \{0\}\subset \mathcal{Z}_A$.
	\end{itemize}
\end{prop}

\subsection{Tiles in $G$}
Recall that a subset $A\subset G$ is said to be a \textit{tile} if there is a set $T\subset G$ such that 
$$
\bigsqcup_{t\in T} (A+t)=G.
$$
In such case, the set $T$ is called a \textit{tiling complement} of $A$ and the pair $(A,T)$ is called a \textit{tiling pair}. Using Fourier transform, we have the following equivalent definitions of tiles in $G$.

\begin{lem}\label{lem:tiling equivalent}
Let $A,B$ be subsets in $G$. Then the following statements are equivalent. 
	\begin{itemize}
		\item [(1)] $(A,B)$ is a tiling pair.
		\item [(2)] $(B,A)$ is a tiling pair.
		\item [(3)] $A\oplus B=G$.
		\item [(4)] $\sharp A\cdot\sharp B=\sharp G$ and $\ZZ_A \cup \ZZ_B = G\setminus \{0\}$. 
	\end{itemize}
\end{lem}
\begin{proof}
	It is trivial to see that $(1) \Leftrightarrow (2) \Leftrightarrow (3)$. It is sufficient to prove $(1) \Leftrightarrow (4)$. In fact,  the pair $(A,B)$ being a tiling pair is equivalent to
	$$\sum_{b\in B} 1_A(x-b) =1, ~\forall~x\in G,$$
	that is to say,
	\begin{equation}\label{eq:AB tile}
	1_A * 1_B =1.
	\end{equation}
	Taking Fourier transformation of both sides of $(\ref{eq:AB tile})$, we get
	$$
	\widehat{1}_A \cdot \widehat{1}_B=(\sharp G) \delta_0,
	$$
	which implies $\sharp A\cdot\sharp B=\sharp G$ and $\ZZ_A \cup \ZZ_B = G\setminus \{0\}$.
\end{proof}

\section{Equi-distribution on planes}\label{sec:Equi-distribution on planes}

In this section, we prove an equi-distributed property for a set in $\Zpp$ ($n,m\in \N$) whenever $\ZZ_A$ is not empty.

Without loss of generality, we assume in this section that $n>m$. Let $\xi_n=e^{2\pi i \frac{1}{p^n}}$ be a primitive $p^n$-th root of unity. For $\vd=(d_1,d_2), \vd'=(d_1', d_2') \in \Zpp$, we define the inner product in $\Zpp$ by the formula
$$
\langle \vd, \vd' \rangle=d_1d_1'+p^{n-m}d_2d_2' \in \Z_{p^n}.
$$
We define 
$$
H_{n,m}(\textbf{d},t):=\{\textbf{x}\in \Zpp: \langle \textbf{x}, \textbf{d} \rangle=t \},
$$
for $\textbf{d}\in \Zpp$ and $t\in \Z_{p^n}$. We call such set \textit{a plane} in $\Zpp$.

For $r\in \Z_{p^n}^*$ and ${\rm\mathbf{d}}=(d_1,d_2)\in \Zpp$, we define the scalar product by
$$
r\vd=(\tilde{d_1}, \tilde{d_2})\in \Zpp,
$$
where $\tilde{d_1}\equiv rd_1 \mod \Z_{p^n}$ and $\tilde{d_2}\equiv rd_2 \mod \Z_{p^m}$.

Now we state the main result in this section.
\begin{lem}\label{lem:hyperplane}
	Let $A\subset \Zpp$ and ${\rm\mathbf{d}}\in \Zpp$. The following are equivalent.
	\begin{itemize}
		\item[(1)] $\hat{1}_A({\rm\mathbf{d}})=0$.
		\item[(2)] $\hat{1}_A(r{\rm\mathbf{d}})=0$ for any $r\in \Z_{p^n}^*$.
		\item[(3)] $\sharp (A\cap H({\rm\mathbf{d}}, t))=\sharp (A\cap H({\rm\mathbf{d}}, t'))$ if $p^{n-1}\mid t-t'$.	
	\end{itemize}
\end{lem}
\begin{proof}
	$(2) \Rightarrow (1):$ It is trivial.

	\medskip

	$(1)\Rightarrow (2):$ Denote by $\vd=(d_1, d_2)$. The fact that $\hat{1}_A({\rm\mathbf{d}})=0$ is equivalent to
	$$
	\sum_{(a_1,a_2)\in A} e^{2\pi i (\frac{a_1d_1}{p^n}+\frac{a_2d_2}{p^m}  )}=0,
	$$
	that is to say,
	$$
	\sum_{(a_1,a_2)\in A} e^{\frac{2\pi i( a_1d_1+p^{n-m}a_2d_2)}{p^n}}=0.
	$$
	It follows from Galois theory that 
	$$
	\sum_{(a_1,a_2)\in A} e^{\frac{2\pi i r(a_1d_1+p^{n-m}a_2d_2)}{p^n}}=0,~\text{for any}~r\in \Z_{p^n}^*,
	$$
	which means that
	$$
	\sum_{(a_1,a_2)\in A} e^{2\pi i (\frac{a_1(rd_1)}{p^n}+\frac{a_2(rd_2)}{p^m}  )}=0,~\text{for any}~r\in \Z_{p^n}^*.
	$$
	We obtain that $\hat{1}_A(r{\rm\mathbf{d}})=0$ for any $r\in \Z_{p^n}^*$.
	\medskip
	
	$(1)\Leftrightarrow (3):$ It is easy to see that 
	$$
	\hat{1}_A(\textbf{d})=\frac{1}{p^{n+m}} \sum_{\textbf{x}\in \Zpp} e^{\frac{2\pi i \langle \textbf{x},\textbf{d} \rangle}{p^n}}1_A(\textbf{x}).
	$$
	It follows that 
	$$
	0=\sum_{\textbf{x}\in \Zpp} \xi_n^{\langle \textbf{x},\textbf{d} \rangle}1_A(\textbf{x})
	=\sum_{t\in \Z_{p^n}} \xi_n^{t}\sum_{\langle \textbf{x},\textbf{d} \rangle=t} 1_A(\textbf{x}).
	$$
	Let 
	$$
	n(t)=\sharp (A\cap H_{n,m}(\vd, t))=\sum_{\langle \textbf{x},\textbf{d} \rangle=t} 1_A(\textbf{x})\in \Z.
	$$
	Then we have
	$$
	\sum_{t\in \Z_{p^n}} n(t) \xi_n^{t}=0.
	$$
	It follows that $\xi_n$ is a root of the polynomial
	$$
	P(X)=\sum_{t\in \Z_{p^n}} n(t) X^{t}\in \Z[X].
	$$
	Recall that the minimal polynomial over $\Z$ of $\xi_n$ is 
	$$
	\Phi_{p^n}(X)=\Phi_p(X^{p^{n-1}})=1+X^{p^{n-1}}+X^{2p^{n-1}}+\cdots+ X^{(p-1)p^{n-1}}.
	$$
	Therefore there exists $Q\in \Z[X]$ such that 
	$$
	P(X)=Q(X)\Phi_{p^n}(X).
	$$
	Hence, we conclude that $n(t)=n(t')$ if $p^{n-1}\mid t-t'$ 
\end{proof}

\begin{rem}\label{rem:type of zeros}
	Now we are concerned with the case $\Zptp$. By Lemma \ref{lem:hyperplane} (2), we define the relation $\vd \sim \vd'$ if there exists $r\in \Z_{p^2}^*$ such that $\vd=r\vd'$.
	Thus the equivalent classes in $\Zptp$ by $\sim$ are
	$$
	(0,1), (1,0), (p,0), (p,c) ~\text{and}~(1,c)~\text{for all}~c\in \Z_p\setminus\{0\}.
	$$ 
	Thus when we study the set of zeros $\mathcal{Z}_{\Zptp}$, we only need to consider the elements which have the above forms.
\end{rem}

\section{Compatible values}\label{sec:Compatible values}
In this section, we introduce the notion of the compatible value for $p$ pairs. This notion will be helpful to deduce some properties of spectral sets and tiles in $\Zptp$, which will be shown in Section \ref{sec:Properties of spectral sets and tiles}.

Given $p$ pairs $(x_i,y_i)_{0\le i\le p-1} \in (\Z_p\times \Z_p)^p$, we consider the value $c\in \Z_p$ such that
$$
x_i+cy_i\not\equiv x_j+cy_j \mod p, ~\text{for all}~i\not=j,
$$
that is to say,
\begin{equation}\label{eq: linear c}
\{ x_0+cy_0, x_1+cy_1, \cdots, x_{p-1}+cy_{p-1}\}\equiv\{ 0,1, \dots, p-1\} \mod p.
\end{equation}
If it is the case, we call the value $c$ \textit{compatible} with the $p$ pairs $(x_i,y_i)_{0\le i\le p-1}$. 

It is not hard to see that for every $a\in \Z_p$, each element in $\Z_p$ is compatible with $(i,a)_{0\le i\le p-1}$. However, given some special $(x_i,y_i)_{0\le i\le p-1}$, we will see that some element in $\Z_p$ can not be compatible with  $(x_i,y_i)_{0\le i\le p-1}$. 
\begin{lem}\label{lem:existence of not compatible}
	Let $(x_i,y_i)_{0\le i\le p-1} \in (\Z_p\times \Z_p)^p$. Suppose that $x_i\not=x_j$ for all $0\le i\not=j\le p-1$. Suppose that not all $y_j$ are the same for $0\le j\le p-1$. Then there exists $c\in \Z_p^*$ such that $c$ is not compatible with  $(x_i,y_i)_{0\le i\le p-1}$.
\end{lem}
\begin{proof}
	Without loss of generality, we assume that $x_i=i$ for all $0\le i\le p-1$. Let $j, k\in \Z_p$ such that $y_j=\alpha$ and $y_k=\beta$ with $\alpha\not=\beta$. Let 
$$
u=(k-j)(\alpha-\beta)^{-1} \in \Z_p.
$$	
	It follows that
	$$
	x_j+uy_j\equiv j+(k-j)(\alpha-\beta)^{-1}\cdot \alpha \equiv (k\alpha-j\beta)(\alpha-\beta)^{-1}  \mod p,
	$$
	and 
	$$
	x_k+uy_k\equiv k+(k-j)(\alpha-\beta)^{-1}\cdot \beta \equiv (k\alpha-j\beta)(\alpha-\beta)^{-1} \mod p.
	$$
	It implies that 
	$$x_j+uy_j \equiv x_k+uy_k \mod p.$$
	Hence, the value $u$ is not compatible with  $(x_i,y_i)_{0\le i\le p-1}$.
\end{proof}

The following lemma tells that the optimum in Lemma \ref{lem:existence of not compatible} can be achieved in some case.
\begin{lem}\label{lem:compatible p-1}
	Let $r\in \Z_p^{*}$. Then the value $c\in \Z_p$ is compatible with $(ri,i)_{0\le i\le p-1}$ if and only if $c\not=p-r$.
\end{lem}
\begin{proof}
	We first observe that 
	$$
	ri+(p-r)i\equiv 0 \mod p, 
	$$
	for all $0\le i\le p-1$. Thus the value $(p-r)$ is not compatible with $(ri,i)_{0\le i\le p-1}$.
	
	On the other hand, for $c \in \Z_p \setminus \{ p-r\}$, we have $(r+c)\in \Z_p^*$ and 
	$$
	ri+ci\equiv (r+c)i \mod p.
	$$ 
	It follows that
	$$
	\{(r+c)i \}_{0\le i\le p-1}\equiv\{0,1,\dots, p-1 \} \mod p.
	$$
	Thus, for any $c\in \Z_p \setminus \{ p-r\}$, it is compatible with $(ri,i)_{0\le i\le p-1}$.
\end{proof}

\section{Properties of spectral sets and tiles}\label{sec:Properties of spectral sets and tiles}

In this section, we show some properties of spectral sets and tiles in $\Zptp$. These properties not only have their own interest but also are useful to the proof of Theorem \ref{main thm}. The main results in this section are summarized as follows. Let $A\subset \Zptp$ with $1<\sharp A<p^3$ (we don't consider the trivial situation where $\sharp A=1$ or $p^3$). Actually, we have already proved some condition on the cardinality of tiles and spectral sets in previous sections:  if $A$ is a tile in $\Zptp$, then by Lemma \ref{lem:tiling equivalent}, $\sharp A$ divides $\sharp (\Zptp)=p^3$, that is, $\sharp A$ has to be $1, p, p^2$ or $p^3$; if $A$ is a spectral set in $\Zptp$, then by Lemma \ref{lem:hyperplane}, $\sharp A$ is divided by $p$. In this section, we first show that if $p^2<\sharp A<p^3$, then the set $A$ has no chance to be a spectral set (Lemma \ref{lem:spectral set larger than p^2}). We then prove that if $\sharp A=p$ and $\ZZ_A\not=\emptyset$, then $A$ is a tile and meanwhile a spectral set (Proposition \ref{prop:sharp A=p}). In the rest of this section, we deal with some special cases: Lemma \ref{lem:(p,0),(1,c)} and Lemma \ref{lem:(p,0),(1,c) tile and spectral} tells us that if some special elements are contained in $\ZZ_A$, then $\sharp A=p^2$ and meanwhile $A$ is a tile and a spectral set; Lemma \ref{lem:subgroup} shows that if a tile (resp. a spectral set) is in a proper subgroup of $\Zptp$, then it is a spectral set (resp. a tile).  Thus for tiles, we only need to focus on the case where $\sharp A=p^2$; for spectral sets, we only need to consider the case where $p<\sharp A\le p^2$. We leave these cases to the next two sections.

For $A\subset \Zptp$, we are concerned with the difference set $(A-A)$ in the following lemma. We show that if $\sharp A$ is large (i.e. $\sharp A>p^2$), then each non-trivial direction is contained in $(A-A)$.
\begin{lem}\label{lem:vanishing all direction}
	Let $A\subset \Zptp$. Suppose $\sharp A>p^2$. Then for all $\vd\in \Zptp\setminus\{0\}$, we have $(A-A)\cap \vd (\Zptp)^*\not=\emptyset$.
\end{lem}
\begin{proof}
	Let $\vd=(c_1+c_2p,c_3)\in (\Zptp)\setminus\{0\}$ where $c_i\in \Z_p$ and at least one of $c_i$ is non-zero for $i=1,2,3$. 
	We first consider the case where $c_1\not=0$. Let 
	$$
	\textbf{v}=(p,0)~\text{and}~\textbf{w}=(0,1).
	$$
	It is not hard to see that every element in $\Zptp$ is a $\Z_p$-combination of $\vd, \textbf{v}$ and $\textbf{w}$, that is, for any $\textbf{x}\in \Zptp$, there exists $a_1, a_2$ and $a_3$ in $\Z_p$ such that
	$$
	\textbf{x}=a_1\vd+a_2\textbf{v}+a_3\textbf{w}.
	$$ 
	Since $\sharp A>p^2$, by pigeonhole principle, there exist $\textbf{x}$ and $\textbf{x}'$ in $\Zptp$ such that
	$$
	\textbf{x}=a_1\vd+a_2\textbf{v}+a_3\textbf{w}~\text{and}~\textbf{x}'=a_1'\vd+a_2\textbf{v}+a_3\textbf{w}
	$$
	with $a_1, a_1', a_2, a_3\in \Z_p$ and $a_1\not=a_1'$. Thus we obtain that $$\textbf{x}-\textbf{x}'=(a_1-a_1')\vd \in \vd (\Zptp)^*.$$
	The proof of the case where $c_2\not=0$ or $c_3\not=0$ is similar. Thus we complete the proof.
\end{proof}

The following lemma shows that for a set $A\subset \Zptp$, if $p^2<\sharp A<p^3$, then $A$ is not a spectral set in $\Zptp$.
\begin{lem}\label{lem:spectral set larger than p^2}
		Let $A\subset \Zptp$ be a spectral set. Suppose $\sharp A>p^2$. Then $A=\Zptp$.
\end{lem}
\begin{proof}
	Let $B$ be a spectrum of $A$. It follows that $\sharp B=\sharp A>p^2$ and 
	$$
	(B-B)\setminus \{ 0\} \subset \ZZ_A.
	$$
	By Lemma \ref{lem:vanishing all direction} and Lemma \ref{lem:hyperplane} $(2)$, we have 
	$$
	\ZZ_A=(\Zptp)\setminus \{\textbf{0}\},
	$$
	which means that 
	$$
\widehat{1}_A(\textbf{x})=0, ~\text{for all}~\textbf{x}\not=0.
	$$
	By elementary Fourier analysis, we have
	$$
	{1}_A(\textbf{x})=\sum_{\vd\in \Zptp} \xi_n^{\langle \textbf{x},\textbf{d} \rangle} \widehat{1}_A(\textbf{x}),
	$$
	where $\xi_n$ is a primitive $n$-th root of unity. It follows that
	$$
	{1}_A(\textbf{0})=\widehat{1}_A(\textbf{0})=\frac{\sharp A}{p^3}.  
	$$
	Since ${1}_A(x)$ is an indicator of a set, implying that ${1}_A(x)$ takes value in $\{0,1\}$, we conclude that $\sharp A=p^3$. Thus, we obtain that the set $A$ has to be $\Zptp$.
\end{proof}

Generally, Lemma \ref{lem:vanishing all direction} and Lemma \ref{lem:spectral set larger than p^2} are also valid for all $2$-dimensional finite abelian groups $\Z_{p^n}\times \Z_{p^m}$  ($n,m\ge 0$):

\vspace{8pt}
{\em Let $A\subset \Z_{p^n}\times \Z_{p^m}$. Suppose $\sharp A>p^{n+m-1}$. Then for all $\vd\in (\Zpp)\setminus\{\textbf{0}\}$, we have $(A-A)\cap \vd (\Z_{p^n}\times \Z_{p^m})^*\not=\emptyset$. Moreover, if $A$ is a spectral set, then $A=\Z_{p^n}\times \Z_{p^m}$.}
\vspace{8pt}

The proof is similar with the proof of Lemma \ref{lem:vanishing all direction} and Lemma \ref{lem:spectral set larger than p^2}. We leave it to the readers to work out the details.

In what follows, we write simply $H(\textbf{d},t)=H_{2,1}(\textbf{d},t)$ for every $t\in \Z_{p^2}$ and $\vd\in \Zptp$. Moreover, for any subset $B\subset \Zptp$, we write 
$$
h_{B}(\vd, t)=\sharp (B\cap H(\vd, t)).
$$
A simple fact which we will use several times later is that the statement $H((p,c),t)\cap B=\emptyset$ for all $c\in \Z_p$ and all $t\in \Z_{p^2}$ with $p\not|t$ implies that $h_B((p,c),t)=0$ for all $c\in \Z_p$ and all $t\in \Z_{p^2}$ with $p\not|t$.

Now we are concerned with the case where $\sharp A=p$.
\begin{prop}\label{prop:sharp A=p}
	Let $A\subset \Zptp$ with $\sharp A=p$. Suppose that $\ZZ_A\not=\emptyset$. Then $A$ is a tile and also a spectral set.
\end{prop}
\begin{proof}
	We first prove that $A$ is a spectral set by constructing its spectrum. Suppose that $\vd\in \ZZ_A$. Define 
	$$
	B:=\{r\vd: 0\le r\le p-1 \}.
	$$
	Since $\vd\not=\textbf{0}$, we have $\sharp B=p$.
	It is easy to see that
	$$
	B-B=\vd \Z_p^* \cup \{0\} \subset \vd \Z_{p^2}^* \cup \{0\}
	$$
	It follows that 
	$$
	(B-B)\setminus \{0\} \subset \ZZ_A.
 	$$
 By Lemma \ref{Equivalence of spectral set}, we obtain that $B$ is a spectrum of $A$.

	It remains to prove that $A$ is a tile. Denote by 
	\begin{equation}\label{eq:form A}
	A=\{(x_i+py_i, z_i): x_i,y_i,z_i\in \Z_p, 0\le i\le p-1 \}.
	\end{equation}
	Due to Remark \ref{rem:type of zeros}, we decompose the proof in the following four cases. \\
	
\textbf{Case 1}: $(0,1)\in \ZZ_A$.	We define 
$$
B_1=\{(x,0): x\in \Z_{p^2} \}.
$$
We will prove that $B_1$ is a tiling complement of $A$. It is easy to see that $\sharp B_1=p^2$.
A simple computation shows that for all $c\in \Z_p$ and for all $t\in \Z_{p^2}$, we have
\begin{equation}\label{eq:B_1(0)}
h_{B_1}((1,c), t+pj)=1~\text{for all}~0\le j\le p-1.
\end{equation}
By Lemma \ref{lem:hyperplane}, we have 
\begin{equation}\label{eq:B_1(1)}
(1,c)\in \ZZ_{B_1},~\forall~c\in \Z_p.
\end{equation}
On the other hand, by similar computation with $(\ref{eq:B_1(0)})$, we obtain that for all $c\in \Z_p$ and $z\in \Z_{p^2}$,
$$
h_{B_1}((p,c), t)=
\begin{cases}
p,~&\text{if}~p\mid t;\\
0,&\text{otherwise}.
\end{cases}
$$
By Lemma \ref{lem:hyperplane}, we have 
\begin{equation}\label{eq:B_1(2)}
(p,c)\in \ZZ_{B_1},~\forall~c\in \Z_p.
\end{equation}
Combining (\ref{eq:B_1(1)}), (\ref{eq:B_1(2)}) and Remark \ref{rem:type of zeros}, we get
$$
\ZZ_A \cup \ZZ_{B_1}=(\Zptp)\setminus\{\textbf{0}\}.
$$
By Lemma \ref{lem:tiling equivalent}, we conclude that $B_1$ is a tiling complement of $A$.

\vspace{8pt}	
\textbf{Case 2}: $\exists c\in \Z_p$ such that $(1,c)\in \ZZ_A$. By Lemma \ref{lem:hyperplane} and the fact that $\sharp A=p$, we obtain that there exists $\bar{x}\in \Z_p$ such that 
\begin{equation}\label{eq: (1,c)}
h_A((1,c),\bar{x}+pj)=1~\text{for all}~0\le j\le p-1.
\end{equation}
Using (\ref{eq:form A}), we rewrite (\ref{eq: (1,c)}) as follows (taking a permutation of the index set $\{i\}_{0\le i\le p-1}$ if necessary):
$$
x_i+py_i+cpz_i\equiv \bar{x}+pi \mod p^2,~\forall~0\le i\le p-1.
$$
It follows that 
$$
x_i=\bar{x}~\text{and}~ y_i+cz_i\equiv i\mod p, ~\forall~0\le i\le p-1.
$$
Thus $A$ has to be the form
$$
\{(\bar{x}+py_i,z_i): y_i,z_i\in \Z_p, y_i+cz_i\equiv i\mod p, 0\le i\le p-1  \}.
$$
Let 
$$
\widetilde{A}=\{(y_i, z_i): 0\le i\le p-1 \} \subset \Z_p\times \Z_p.
$$
Following the results in \cite{imp}, the set $\widetilde{A}$ is a tile. Denoting by $\widetilde{B}_2$ a tiling complement of $\widetilde{A}$, we define
$$
B_2=\{(x+py,z): x\in \Z_p, (y,z)\in \widetilde{B}_2 \}.
$$ 
It follows that $B_2$ is a tiling complement of $A$. In fact, for any $(x+py,z)\in \Zptp$, there exists unique decomposition 
$$
(x+py,z)=(\bar{x}+py',z')+((x-\bar{x})+py'',z''),
$$
where $(\bar{x}+py',z')\in A$, $((x-\bar{x})+py'',z'')\in B_2$ and $(y,z)=(y',z')+(y'',z'')\in \Z_p\times \Z_p$ in the unique way according to the tiling pair $(\widetilde{A}, \widetilde{B})$.

\vspace{8pt}
\textbf{Case 3}: $\exists c\in \Z_p^*$ such that $(p,c)\in \ZZ_A$.
We define 
$$
B_3=\{(x+py,(p-c)x): x,y\in \Z_p \}.
$$
It is easy to see that $\sharp B_3=p^2$. The statement that the set $B_3$ is a tiling complement of $A$ is a direct consequence by Lemma \ref{lem:tiling equivalent} and the following claim. 

\vspace{8pt}
\textbf{Claim:} $\ZZ_{B_3}=(\Zptp)\setminus\{(0,0), (p,c) \}$.
\begin{proof}
	According to Remark \ref{rem:type of zeros}, we will prove the following three kinds of zeros are belonging to $\ZZ_{B_3}$.
	
	\vspace{6pt}
	\textbf{1.} $(0,1)$. The result is followed directly from Lemma \ref{lem:hyperplane} and the fact that for $z\in \Z_{p^2}$,
	$$
	h_{B_3}((0,1), t)=
	\begin{cases}
	p,~&\text{if}~p\mid t;\\
	0,&\text{otherwise}.
	\end{cases}
	$$
	Thus we get $(0,1)\in \ZZ_{B_3}$.

	\vspace{6pt}
	\textbf{2.} $(1,d)$ for all $d\in \Z_p$. A simple calculation shows that for all $d\in \Z_p$ and for all $t\in \Z_{p^2}$, we have
	$$
	h_{B_3}((1,d), t)=1.
	$$
	By Lemma \ref{lem:hyperplane}, we obtain that $(1,d)\in \ZZ_{B_3}$ for all $d\in \Z_p$.
	
	\vspace{6pt}
	\textbf{3.} $(p,d)$ for all $d\in \Z_p\setminus\{c\}$. By Lemma \ref{lem:compatible p-1}, we obtain that for all $d\in \Z_p\setminus\{c\}$ and for $t\in \Z_{p^2}$,
	$$
	h_{B_3}((p,d), t)=
	\begin{cases}
	p,~&\text{if}~p\mid t;\\
	0,&\text{otherwise}.
	\end{cases}
	$$
	By Lemma \ref{lem:hyperplane}, we have $(p,d)\in \ZZ_{B_3}$ for all $d\in \Z_p\setminus\{c\}$.
\end{proof}

\vspace{8pt}	
\textbf{Case 4}: $(p,0)\in \ZZ_A$. Define 
$$
B_4=\{(py,z): y,z\in \Z_p \}.
$$
We will prove that $B_4$ is a tiling complement of $A$. It is easy to see that $\sharp B_4=p^2$.
A simple computation shows that for all $c\in \Z_p$, we have
$$
h_{B_4}((1,c), t)=
\begin{cases}
p,~&\text{if}~p\mid t;\\
0,&\text{otherwise}.
\end{cases}
$$
By Lemma \ref{lem:hyperplane}, we have 
\begin{equation}\label{eq:B_4(1)}
(1,c)\in \ZZ_{B_4},~\forall~c\in \Z_p.
\end{equation}
On the other hand, we obtain that for all $c\in \Z_p^*$,
$$
h_{B_4}((p,c), t)=
\begin{cases}
p,~&\text{if}~p\mid t;\\
0,&\text{otherwise}.
\end{cases}
$$
Moreover, we have
$$
h_{B_4}((0,1), t)=
\begin{cases}
p,~&\text{if}~p\mid t;\\
0,&\text{otherwise}.
\end{cases}
$$
By Lemma \ref{lem:hyperplane}, we have 
\begin{equation}\label{eq:B_4(2)}
(p,c), (0,1)\in \ZZ_{B_4},~\forall~c\in \Z_p^*.
\end{equation}
Combining (\ref{eq:B_4(1)}), (\ref{eq:B_4(2)}) and Remark \ref{rem:type of zeros}, we get
$$
\ZZ_A \cup \ZZ_{B_1}=(\Zptp)\setminus\{\textbf{0}\}.
$$
By Lemma \ref{lem:tiling equivalent}, we conclude that $B_4$ is a tiling complement of $A$.
	
\vspace{8pt}

To sum up, we conclude that the set $A$ is a tile in $\Zptp$.
	
\end{proof}

According to Remark \ref{rem:type of zeros}, we only need to study some particular zeros. The following lemma shows that if $\ZZ_B$ contains two special kinds of zeros, then $\sharp B=p^2$ and $B$ has a certain form. 

\begin{lem}\label{lem:(p,0),(1,c)}
	Let $B$ be a subset of $\Zptp$. Suppose that $\sharp B\le p^2$ and $(p,0),(1,c)\in \ZZ_B$ for some $c\in \Z_p$. Then $\sharp B=p^2$ and $B$ has the form 
	\begin{equation}\label{eq:(p,0),(1,c)}
	\{(i+py_{ij}, z_{ij}): y_{ij}+cz_{ij}\equiv j \mod p, 0\le i,j\le p-1 \}.
	\end{equation}
\end{lem}
\begin{proof}
	By the fact that $(p,0)\in \ZZ_B$ and Lemma \ref{lem:hyperplane}, we have
	$$
	h_B((p,0), pj)=\frac{1}{p}\sharp B~\text{for all}~0\le j\le p-1.
	$$ 
	We observe that 
	\begin{equation}\label{eq:(p,0)(1,c) 1}
	H((1,c), i+pj)\subset H((p,0),pi),~\text{for all} ~0\le i,j\le p-1.
	\end{equation}
	In fact, supposing $(x+py,z)\in H((1,c), i+pj)$, we obtain that
	\begin{equation}\label{eq:(p,0)(1,c) 3}
	x+py+pcz\equiv i+pj \mod p^2,
	\end{equation}
	which implies
	$$
	x\equiv i \mod p,
	$$
	that is to say, $(x+py,z)\in H((p,0),pi)$. By the fact that $(1,c)\in \ZZ_B$ and Lemma \ref{lem:hyperplane}, we have 
	\begin{equation}\label{eq:(p,0)(1,c) 2}
	h_B((1,c), i)=h_B((1,c), i+p)=\cdots=h_B((1,c), i+p(p-1)),
	\end{equation}
	for all $0\le i\le p-1$.
	Combining (\ref{eq:(p,0)(1,c) 1}) and (\ref{eq:(p,0)(1,c) 2}), we obtain that $\frac{1}{p}\sharp B$ is divided by $p$ and 
	$$
	h_B((1,c), i+pj)=\frac{1}{p} h_B((p,0),pi),~\text{for all} ~0\le i,j\le p-1.
	$$
	Since $\sharp B\le p^2$, we have $\sharp B=p^2$ and $h_B((1,c), i+pj)=1$ for all $0\le i,j\le p-1$. By (\ref{eq:(p,0)(1,c) 3}), the point $(x+py,z)$ lies in the plane $ H((1,c), i+pj)$ if and only if 
	$$
	x\equiv i \mod p~\text{and}~y+cz\equiv j \mod p.
	$$
	Denoting by $(i+py_{ij}, z_{ij})$ the only point lying in $H((1,c), i+pj)\cap B$,  we conclude that the set $B$ has to have the form (\ref{eq:(p,0),(1,c)}).
\end{proof}

We show that the set having the form  in Lemma \ref{lem:(p,0),(1,c)} is exactly a tile and also a spectral set.
\begin{lem}\label{lem:(p,0),(1,c) tile and spectral}
	Let $B$ be a subset of $\Zptp$. Suppose that $B$ has the form (\ref{eq:(p,0),(1,c)}). Then $B$ is a tile and also a spectral set.
\end{lem}
\begin{proof}
	Let 
	$$
	A=\{(j+pi, cj): 0\le i,j\le p-1 \}.
	$$
	It is easy to see that $\sharp A=p^2$. We will prove that $A$ is a spectrum of $B$. In fact, a simple computation shows that
	$$
	(A-A)\setminus\{0\}\subset (1,c) (\Zptp)^*\cup (p,0)(\Zptp)^*.
	$$
	Since $(1,c),(p,0)\in \ZZ_B$, by Lemma \ref{Equivalence of spectral set}, we conclude that $A$ is a spectrum of $B$.

	It remains to show that $B$ is also a tile. Let $\Omega$ be a subset of $\Z_p\times \Z_p$ which has the form
	$$
	\{(y_j,z_j): y_{j}+cz_{j}\equiv j \mod p, 0\le j\le p-1 \}.
	$$
	According to \cite{imp}, we see that $\Omega$ is a tile. Let $T$ be a tiling complement of $\Omega$. Let 
	$$
	C=\{(py,z): (y,z)\in T\}.
	$$
	 We will prove that $C$ is actually a tiling complement of $B$. In fact, for any $(x+py,z)\in \Zptp$, there exists the unique decomposition 
$$
(x+py,z)=(x+py',z')+(py'',z''),
$$
where $(x+py',z')\in B$, $(py'',z'')\in C$ and $(y,z)=(y',z')+(y'',z'')\in \Z_p\times \Z_p$ in the unique way according to the tiling pair $(\Omega, T)$.

\end{proof}

Lemma \ref{lem:(p,0),(1,c)} and Lemma \ref{lem:(p,0),(1,c) tile and spectral} gives us a sufficient condition for tiles (or spectral sets) and is useful to handle the case where $p<\sharp B\le p^2$, which will be shown in Section \ref{sec:Tiles implying Spectral sets} and Section \ref{Spectral sets implying Tiles}.

We finish this section by the following lemma concerning tiles or spectral sets in a proper subgroup of $\Zptp$.
\begin{lem}\label{lem:subgroup}
	Let $G$ be a proper subgroup of $\Zptp$. Let $A\subset G$. Then $A$ is a tile if and only if it is a spectral set.
\end{lem}
\begin{proof}
	It is not hard to check that up to an isomorphism, the proper subgroups of $\Zptp$ are 
	$$
	\Z_p, \Z_p\times \Z_p~\text{and}~\Z_{p^2}.
	$$
	We know that Fuglede's spectral set conjecture holds on these groups (see \cite{l2}, \cite{ffs} and \cite{imp}). Thus we complete the proof.
\end{proof}

\section{Tiles $\Rightarrow$ Spectral sets}\label{sec:Tiles implying Spectral sets}
Let $(A,B)$ be a tiling pair in $\Zptp$. In this section, we will prove that both $A$ and $B$ are spectral sets. 

If $\sharp A=1$ or $\sharp B=1$, then it is easy to see that $A$ and $B$ are spectral sets. Now suppose that $\sharp A>1$ and $\sharp B>1$. Without loss of generality, we may assume that $\sharp A=p$ and $\sharp B=p^2$. Due to Proposition \ref{prop:sharp A=p}, we already know that $A$ is a spectral set. It remains to prove that $B$ is a spectral set. 

We observe that $(1,0)$ and $(p,0)$ can not be in $\ZZ_A$ at the same time. In fact, if $(1,0), (p,0)\in \ZZ_A$, then $p^2$ divides $\sharp A$. This is contradict to the hypothesis $\sharp A=p$. By Lemma \ref{lem:tiling equivalent}, either $(1,0)$ or $(p,0)$ is in $\ZZ_B$. We thus decompose the proof into the following three cases.

\vspace{10pt}
\textbf{Case 1:} $(1,0), (p,0)\in \ZZ_B$. By Lemma \ref{lem:(p,0),(1,c)} and Lemma \ref{lem:(p,0),(1,c) tile and spectral}, we obtain that $B$ is a spectral set.

\vspace{10pt}
\textbf{Case 2:} $(1,0)\in \ZZ_B$ and $(p,0)\notin \ZZ_B$. It follows that $(p,0)\in \ZZ_A$.  By Lemma \ref{lem:(p,0),(1,c)} and the fact that $\sharp A=p$, we have that $(1,a)\notin \ZZ_A$ for all $0\le a\le p-1$. It follows that $(1,a)\in \ZZ_B$ for all $0\le a\le p-1$.

On the other hand, by Lemma \ref{lem:hyperplane} and the fact that $\sharp A=p$, we obtain that 
\begin{equation}\label{eq: (p,0)}
h_A((p,0),0)=h_A((p,0),p)=\dots=h_A((p,0),(p-1)p)=1.
\end{equation}
Denote by
\begin{equation}\label{eq:form A 2}
A=\{(x_i+py_i, z_i): 0\le i\le p-1 \}
\end{equation}
Due to (\ref{eq: (p,0)}), we may assume that 
$$
(x_i+py_i, z_i)\in H((p,0),ip)~\text{for all}~0\le i\le p-1.
$$
It follows that
$$
px_i+p^2y_i\equiv ip \mod p^2,~\forall~0\le i\le p-1.
$$
This implies that $x_i\equiv i \mod p$, for all $0\le i\le p-1.$
Thus $A$ has to be the form
$$
\{(i+py_i,z_i): 0\le i\le p-1  \}.
$$
We consider the following two cases.

\vspace{10pt}
\begin{itemize}
\item[(1)] All $z_i$ are the same for $0\le i\le p-1$.

\vspace{5pt}

It follows that not all $h_A((0,1), p_i)$ are the same for $0\le i\le p-1$. By Lemma \ref{lem:hyperplane}, we have $(0,1)\notin \ZZ_A$. Then $(0,1)\in \ZZ_B$. We define 
$$
\overline{B}=\{(x,z): x,z\in \Z_p \}\subset \Zptp.
$$
It is easy to see that $\sharp B=p^2$ and $(\overline{B}-\overline{B})\setminus\{0\}$ is contained in
$$
(0,1)(\Zptp)^*\bigcup \cup_{a\in \Z_p} (1,a) (\Zptp)^*,
$$
which is a subset of $\ZZ_B.$
Thus we obtain that $\overline{B}$ is a spectrum of $B$.

\vspace{5pt}

\item[(2)] Not all $z_i$ are the same for $0\le i\le p-1$.

\vspace{5pt}

By Lemma \ref{lem:existence of not compatible}, there exists $1\le c\le p-1$ such that the value $c$ is not compatible with $(i, z_i)_{0\le i\le p-1}$. It follows that not all $h_A((p,c), pi)$ are the same for $0\le i\le p-1$. 
Thus we have $(p,c)\notin \ZZ_A$. Then we have $(p,c)\in \ZZ_B$. Define 
$$
\widetilde{B}=\{(x+py,cy): x,y\in \Z_p \}.
$$
It is easy to see that $\sharp B=p^2$.
A simple computation shows that the set $(\widetilde{B}-\widetilde{B})\setminus\{0\}$ is contained in
$$
(p,c)(\Zptp)^*\bigcup \cup_{a\in \Z_p} (1,a) (\Zptp)^*,
$$
which is a subset of $\ZZ_B.$
Thus we conclude that $\widetilde{B}$ is a spectrum of $B$.

\end{itemize}

\vspace{10pt}
\textbf{Case 3:} $(p,0)\in \ZZ_B$ and $(1,0)\notin \ZZ_B$. It follows that $(1,0)\in \ZZ_A$. According to Case 2 in Proposition \ref{prop:sharp A=p}, we obtain that $A$ has the form
$$
\{x+pi,z_i: i\in \Z_p \},
$$
for some $x\in \Z_p$.
If not all $z_i$ are the same for $0\le i\le p-1$, then by Lemma \ref{lem:existence of not compatible}, there exists $c\in \Z_p^*$ such that the value $c$ is not compatible with $(i, z_i)_{0\le i\le p-1}$. It follows that not all $h_A((1,c), x+pi)$ are the same for $0\le i\le p-1$. 
Thus we have $(1,c)\notin \ZZ_A$. By Lemma \ref{lem:tiling equivalent}, we have $(1,c)\in \ZZ_B$. By Lemma \ref{lem:(p,0),(1,c)} and Lemma \ref{lem:(p,0),(1,c) tile and spectral}, we obtain that $B$ is a spectral set. Now suppose that all $z_i$ are the same for $0\le i\le p-1$. It is not hard to check that 
$$
h_A((0,1), pz_i)=p~\text{and}~h_A((p,a),p(x+az_i))=p,
$$ 
for all $0\le a\le p-1$ and all $0\le i\le p-1$.
By Lemma \ref{lem:hyperplane}, we have $(0,1), (p,a)\notin \ZZ_A$ for all $0\le a\le p-1$. Then $(0,1), (p,a)\in \ZZ_B$  for all $0\le a\le p-1$. We define 
$$
\overline{B}=\{(py,z): y,z\in \Z_p \}\subset \Zptp.
$$
It is easy to see that $\sharp B=p^2$ and
$$
(\overline{B}-\overline{B})\setminus\{0\} \subset (0,1)(\Zptp)^*\bigcup \cup_{a\in \Z_p} (p,a) (\Zptp)^*\subset \ZZ_B.
$$
Thus we obtain that $\overline{B}$ is a spectrum of $B$.

\section{Spectral sets $\Rightarrow$ Tiles}\label{Spectral sets implying Tiles}

Let $A$ be a spectral set in $\Zptp$. In this section, we will prove that $A$ is also a tile.

If $\sharp A=1$ or $p^3$, then $A$ is a tile trivially. We suppose that $1<\sharp A<p^3$. By Lemma \ref{lem:spectral set larger than p^2}, we have $\sharp A\le p^2$.
Let $B$ be a spectrum of $A$. By Lemma \ref{Equivalence of spectral set}, we have 
$$
(B-B)\setminus \{0\}\subset \ZZ_A.
$$
Since $\sharp B=\sharp A>1$, we have $\ZZ_A\not=\emptyset$. By Lemma \ref{lem:hyperplane}, we get $\sharp A\ge p$. If $\sharp A=p$, then by Proposition \ref{prop:sharp A=p}, we know that $A$ is a tile. Thus in what follows, we assume that $$p<\sharp A\le p^2.$$
By pigeon principle, we have
\begin{equation}\label{eq:(1,0)(p,0)in ZZ_A}
(1,0)~\text{or}~(p,0)\in (B-B)\setminus \{0\} \subset \ZZ_A.
\end{equation}
The same is for $\ZZ_B$.

Without loss of generality, we assume that $(0,0)\in A$ and $(0,0)\in B$. If $(A-A)\cap (\Z_{p^2}^* \times \Z_p)=\emptyset$ or $(B-B)\cap (\Z_{p^2}^* \times \Z_p)=\emptyset$, then we obtain 
$$
A\subset p\Z_{p^2} \times \Z_p~\text{or}~B\subset p\Z_{p^2} \times \Z_p.
$$
It follows that either $A$ or $B$ is in a non-trivial proper subgroup of $\Zptp$. By Lemma \ref{lem:subgroup}, we deduce that $A$ and $B$ are tiles.

Now we suppose that $(A-A)\cap (\Z_{p^2}^* \times \Z_p)\not=\emptyset$ and $(B-B)\cap (\Z_{p^2}^* \times \Z_p)\not=\emptyset$. This implies that there exists $c_1, c_2\in \Z_p$ such that 
\begin{equation}\label{eq:(1,c)in ZZ_A}
(1,c_1)\in \ZZ_A~\text{and}~(1,c_2)\in \ZZ_B.
\end{equation}

\begin{lem}\label{lem:(p,0) A or B}
	The sets $A$ and $B$ are as above. Then $(p,0)\in \ZZ_A$ or $(p,0)\in \ZZ_B$.
\end{lem}
\begin{proof}
	We prove it by contradiction. We assume that $(p,0)\notin \ZZ_A$ and $(p,0)\notin \ZZ_B$. By (\ref{eq:(1,0)(p,0)in ZZ_A}), we have $(1,0)\in \ZZ_A$  and $(1,0)\in \ZZ_B$. 
	If there are two different elements $(x+py,z), (x+py',z)\in A$, then 
	$$
	(x+py,z)-(x+py',z)=(p(y-y'),0).
	$$
	Since $(y-y')\in \Z_{p^2}^*$, by Lemma \ref{Equivalence of spectral set}, we have $(p,0)\in \ZZ_B$ which is not possible. Thus there exists $E_{A} \subset \Z_p\times \Z_p$ and $f:E_A\to \Z_p$ such that
	$$
	A=\{(x+pf(x,z),z): (x,z)\in E_A \}.
	$$ 
	Observe that 
	$$
	h_A((1,0),x+pi)=\sharp \{z\in\Z_p: (x,z)\in E_A, f(x,z)=i \},
	$$
	for all $0\le i\le p-1$.
	Since $(1,0)\in \ZZ_A$, we have that $\{z\in\Z_p: (x,z)\in E_A, f(x,z)=i \}$ are constant for all $0\le i\le p-1$.
	By Lemma \ref{lem:hyperplane}, we obtain that 
	\begin{equation*}
	E_A=\widetilde{E}_A\times \Z_p~\text{for some}~\widetilde{E}_A\subset \Z_p,
	\end{equation*} 
	and 
	\begin{equation*}\label{eq:f_x}
	\widetilde{f}_{x}(y):=f(x,y)~\text{is surjective, for all}~x\in \widetilde{E}_A.
	\end{equation*}
	This means that
	\begin{equation*}
	A=\{(x+pf(x,z),z): x\in \widetilde{E}_A, z\in \Z_p \}.
	\end{equation*}
	The same is for $B$, that is to say, 
	\begin{equation*}
	B=\{(x+pg(x,z),z): x\in \widetilde{E}_B, z\in \Z_p \},
	\end{equation*}
	for some $\widetilde{E}_B\subset \Z_p$ and $g: \widetilde{E}_B \times \Z_p \to \Z_p$.  Since $\sharp A>p$, we have $\sharp \widetilde{E}_A>1$. It follows that
	\begin{equation*}
	(A-A)\cap (1,c)(\Zptp)^* \not=\emptyset, ~\text{for all}~c\in \Z_p.
	\end{equation*}
	By Lemma \ref{Equivalence of spectral set}, we have 
	$$
	(1,c)\in \ZZ_B~\text{for all}~c\in \Z_p.
	$$
	This implies that $c$ is compatible with $(g(x,z), z)_{0\le z\le p-1}$ for any $x\in \widetilde{E}_B$ and for any $0\le c\le p-1$. For any $x\in \widetilde{E}_B$, 
	since the map $g(x, \cdot): \Z_p \to \Z_p$ is subjective, it is bijective. This is contradict to Lemma \ref{lem:existence of not compatible}. Therefore, we conclude that  $(p,0)\in \ZZ_A$ or $(p,0)\in \ZZ_B$.
\end{proof}

Due to Lemma \ref{lem:(p,0) A or B}, without loss of generality,  we suppose $(p,0)\in \ZZ_B$. Since $(1,c_2)\in \ZZ_B$, by Lemma \ref{lem:(p,0),(1,c)} and Lemma \ref{lem:(p,0),(1,c) tile and spectral}, the set $B$ is a tile and $\sharp B=p^2$. If we also have $(p,0)\in \ZZ_A$, then repeating the same reason, we obtain that $A$ is a tile, which completes the proof. Now we suppose that $(p,0)\notin \ZZ_A$. By the proof of Lemma \ref{lem:(p,0) A or B}, we have
$$
B=\{(x+pg(x,z), z): x,z\in \Z_p \}
$$
for some $g: \Z_p\times \Z_p \to \Z_p$ with $g(x,\cdot)$ bijective for all $x\in \Z_p$. It follows that
$$
(B-B)\cap (1,c)(\Zptp)^* \not=\emptyset~\text{for all}~c\in \Z_p.
$$
By Lemma \ref{Equivalence of spectral set}, we have
$$
(1,c)\in \ZZ_A~\text{for all}~c\in \Z_p.
$$
Since $\sharp A=\sharp B>p$, by pigeon principle, we obtain that 
$$
(0,1)\in \ZZ_A,
$$
or there exists $d\in \Z_p^*$ such that 
\begin{equation*}
(p,d)\in \ZZ_A.
\end{equation*}

If $(p,d)\in \ZZ_A$ for some $d\in \Z_p^*$, then we define
$$
C=\{(j, (-d)^{-1}j): 0\le j\le p-1 \}.
$$
It is easy to see that $\sharp C=p$. Observe that $h_C((0,1), pi)=1$ for all $0\le i\le p-1$. By Lemma \ref{lem:hyperplane}, we have $(0,1)\in \ZZ_C$. On the other hand, by Lemma \ref{lem:compatible p-1}, the value $d'$ is compatible with $(j, (-d)^{-1}j)_{0\le j\le p-1}$ for all $0\le d'\le p-1$ with $d'\not=d$. Then we have $h_C((p,d'), pi)=1$ for all $0\le i\le p-1$ and for all $0\le d'\le p-1$ with $d'\not=d$. By Lemma \ref{lem:hyperplane}, we have $(p, d') \in \ZZ_C$ for all $0\le d'\le p-1$ with $d'\not=d$.
By Lemma \ref{lem:tiling equivalent}, we conclude that $C$ is a tiling complement of $A$.

If $(0,1)\in \ZZ_A$, then we define
$$
D=\{(j,0):0\le j\le p-1  \}.
$$
It is easy to see that $\sharp C=p$ and $h_C((p,d), pi)=1$ for all $0\le i\le p-1$ and all $0\le d\le p-1$.
By Lemma \ref{lem:hyperplane}, we have 
$$
(p, d)\in \ZZ_C~\text{for all}~0\le d\le p-1.
$$
By Lemma \ref{lem:tiling equivalent}, we conclude that $D$ is a tiling complement of $A$.
 This completes the proof.


\end{document}